\documentclass[a4paper,12pt]{article}

\usepackage{amsmath}
\usepackage{amssymb}
\usepackage{amsthm}
\usepackage{latexsym}
\usepackage{amsmath}
\usepackage{amsthm}
\usepackage{graphicx}
\usepackage{txfonts}
\usepackage{bm}
\usepackage{epsfig}
\usepackage{color}

\usepackage{geometry}

\newtheorem{theorem}{Theorem}

\newtheorem{lemma}[theorem]{Lemma}
\newtheorem{proposition}[theorem]{Proposition}
\newtheorem{remark}{Remark}[section]

\newcommand{\RED}[1]{{\color{red}#1}} 
 \renewcommand{\RED}[1]{{#1}}

\newcommand{\memo}[1]{{\bf [MEMO:} #1 \ {\bf :end memo] }}  
   \renewcommand{\memo}[1]{}           
\newcommand{\OMIT}[1]{{\bf [OMIT:} #1 \ {\bf --- end OMIT] }}  
   \renewcommand{\OMIT}[1]{}            

\newcommand{\RR}{{\mathbb{R}}}

\newcommand{\veczero}{{\bf 0}}
\newcommand{\dom}{{\rm dom\,}}

\newcommand{\unitvec}[1]{\chi_{#1}}

\newcommand{\finbox}{\hspace*{\fill}$\rule{0.2cm}{0.2cm}$}
\newcommand{\todaye}{\the\year/\the\month/\the\day}

\newcommand{\Mnat}{{M$^{\natural}$}}

\numberwithin{equation}{section}

\title{On Equivalence of M$\sp{\natural}$-concavity of a Set Function and 
Submodularity of Its Conjugate%
\footnote{
\RED{
This is a revised version of the paper of the same title published in
\textit{Journal of the Operations Research Society of Japan},
{\bf 61} (2018), 163--171.
In particular, the proofs of Lemmas \ref{LMmnatFRconjsubm01-1} 
and \ref{LMmnatFRconjsubm01-4} are improved.
}
}
}

\author{%
Kazuo Murota%
\thanks{%
Faculty of Economics and Business Administration, 
Tokyo Metropolitan University,
Tokyo 192-0397, Japan; 
murota@tmu.ac.jp
}, \and 
Akiyoshi Shioura%
\thanks{%
Department of Industrial Engineering and Economics, 
Tokyo Institute of Technology,
Tokyo 152-8550, Japan;
shioura.a.aa@m.titech.ac.jp
}
}

\date{July 2017 / December 2022}

\begin{document}

\maketitle

\begin{abstract}
A fundamental theorem in discrete convex analysis states that
a set function is M$\sp{\natural}$-concave 
if and only if its conjugate function  is submodular.
This paper gives a new proof to this fact.
\end{abstract}




\section{Introduction}
\label{SCintro}

Let
$f: 2^{N} \to \RR \cup \{ -\infty \}$
be a set function on a finite set $N = \{ 1,2,\ldots, n \}$,
where the effective domain 
$\dom f =  \{ X \subseteq N \mid f(X) > -\infty \}$
is assumed to be nonempty.
The conjugate function $g: \RR\sp{N} \to \mathbb{R}$ of $f$ is defined by
\begin{align} 
 g(p) 
 &= \max\{  f(X) - p(X)   \mid X \subseteq N \}
\qquad ( p \in \RR\sp{N}),
\label{conjcave2vex01}
\end{align}
where $p(X) = \sum_{i \in X} p_{i}$
(see Remark \ref{RMconjdef}).

A set function
$f: 2^{N} \to \RR \cup \{ -\infty \}$
with $\dom f \not= \emptyset$
is called 
{\em M$\sp{\natural}$-concave}
\cite{Mdcasiam,MS99gp} if,
for any $X, Y \in \dom f$ and $i \in X \setminus Y$,
it holds that%
\footnote{
We use short-hand notations such as
$X - i = X \setminus  \{ i \}$,
$Y + i = Y \cup \{ i \}$,
$X - i + j =(X \setminus  \{ i \}) \cup \{ j \}$,
and
$Y + i - j =(Y \cup \{ i \}) \setminus \{ j \}$.
}  
\begin{equation}  \label{mconcav1}
f( X) + f( Y ) \leq f( X - i ) + f( Y + i ),
\end{equation}
or there exists some $j \in Y \setminus X$ such that
\begin{equation}  \label{mconcav2}
f( X) + f( Y ) \leq 
 f( X - i + j) + f( Y + i - j).
\end{equation}
Since $f( X) + f( Y ) > -\infty$ for $X, Y \in \dom f$,
 (\ref{mconcav1}) requires $X - i, Y + i \in \dom f$,
and (\ref{mconcav2}) requires $X - i + j, Y + i - j  \in \dom f$.
A function $g: \mathbb{R}\sp{N} \to \mathbb{R}$
is called {\em submodular}
if it satisfies the following inequality:
\begin{equation} \label{gsubmR}
  g(p) + g(q) \geq g(p \vee q) + g(p \wedge q)
\qquad (p, q \in \RR\sp{N}),
\end{equation}
where
$p \vee q$ and $p \wedge q$ are the componentwise maximum and minimum of $p$ and $q$, respectively.

The following theorem states one of the most fundamental facts
in discrete convex analysis \cite{Mdca98,Mdcasiam} that
M$\sp{\natural}$-concavity of a set function $f$
can be characterized by submodularity of 
the conjugate function $g$.

\begin{theorem} \label{THmnatcavbyconjfn01}
A set function
$f: 2\sp{N} \to \mathbb{R} \cup \{ -\infty  \}$
with $\dom f \not= \emptyset$
is M$\sp{\natural}$-concave 
if and only if
its conjugate function 
$g: \RR\sp{N} \to \mathbb{R}$
is submodular.
\finbox
\end{theorem}

This theorem was first given by Danilov and Lang \cite{DL01gs} in Russian;
it is cited by Danilov, Koshevoy, and Lang \cite{DKL03gs}.
It can also be derived through a combination of 
Theorem 10 of Ausubel and Milgrom \cite{AM02auc} 
with the equivalence of gross substitutability and \Mnat-convexity
due to Fujishige and Yang \cite{FY03gs}.
A self-contained detailed proof can be found in a recent survey paper
by Shioura and Tamura \cite[Theorem 7.2]{ST15jorsj}.

The objective of this paper is to give yet another proof to the above theorem.
The proof does not use polyhedral-geometric characterizations 
of \Mnat-convex sets and functions, nor  does it depend on the M-L conjugacy theorem 
in discrete convex analysis.
Section~\ref{SCprelim} offers preliminaries from discrete convex analysis,
and Section~\ref{SCproof} presents the proof.
Section~\ref{SClocexcproof} is a technical appendix.

\begin{remark} \rm  \label{RMconjdef}
The definition (\ref{conjcave2vex01}) of the conjugate function $g(p)$ here 
is consistent with its interpretation in economics.
If $f(X)$ denotes the utility (or valuation) function for a bundle $X$,
then $g(p)$ in (\ref{conjcave2vex01}) is the indirect utility function under the price vector $p$.
In convex analysis, however, the conjugate of a concave function $f$ 
is more often defined as 
$f\sp{*}(p) = \min_{x} \{  p\sp{\top} x - f(x) \}$.
\finbox
\end{remark}

\section{Preliminaries on M-concave Functions}
\label{SCprelim}

A set function
$f: 2^{N} \to \RR \cup \{ -\infty \}$
with $\dom f \not= \emptyset$
is called 
{\em valuated matroid} \cite{DW90, DW92}
if,
for any $X, Y \in \dom f$ and $i \in X \setminus Y$,
there exists some $j \in Y \setminus X$ such that
\begin{equation}  \label{valmatexc1}
f( X) + f( Y ) \leq 
 f( X - i + j) + f( Y + i -j).
\end{equation}
This property is referred to as the {\em exchange property}.
A valuated matroid is also called an {\em M-concave set function}
\cite{Mstein96,Mdcasiam}.
The effective domain $\mathcal{B}$ of an M-concave function
forms the family of bases of a matroid,
and in particular,
$\mathcal{B}$ consists of equi-cardinal subsets,
i.e., $|X| = |Y|$ for all $X, Y \in \mathcal{B}$.

As is obvious from the definitions, M-concave functions form a subclass of
M$\sp{\natural}$-concave functions.

\begin{proposition}  \label{PRmcav=mnatcav+equicard}
A set function $f$ is M-concave
 if and only if it is an M$\sp{\natural}$-concave function and 
$|X| = |Y|$ for all $X, Y \in \dom f$.
\finbox
\end{proposition}

The concepts of M-concave and M$\sp{\natural}$-concave functions
are in fact equivalent.
For a function
$f: 2^{N} \to \RR \cup \{ -\infty \}$,
we associate a function $\tilde{f}$
with an equi-cardinal effective domain. 
Denote by $r$ and $r'$ the maximum and minimum, respectively, of $|X|$ for $X \in \dom f$.
Let
$s \geq r-r'$,
$S = \{ n+1,n+2,\ldots, n+s \}$,
and
$\tilde{N} = N \cup S = \{ 1,2,\ldots, \tilde n \}$,
where $\tilde n =n+s$.
We define
$\tilde{f}: 2^{\tilde N} \to \RR \cup \{ -\infty \}$ by
\begin{align} \label{assocMdef} 
\tilde{f}(Z)  =
   \left\{  \begin{array}{ll}
    f(Z \cap N)         &   (|Z| = r) ,     \\
   -\infty    &   (\mbox{otherwise}) .  \\
                     \end{array}  \right.
\end{align}
Then, for $X \subseteq N$ and $U \subseteq S$,
we have $\tilde{f}(X \cup U) = f(X)$ 
if $|U|=r - |X|$.

\begin{proposition}  \label{PRmnatequicardvalmat}
A set function $f$ is M$\sp{\natural}$-concave 
if and only if $\tilde{f}$ is M-concave.
\end{proposition}
\begin{proof}
This fact is well known among experts.
Since $f$ is a projection of $\tilde{f}$, the ``if'' part 
follows from \cite[Theorem 6.15 (2)]{Mdcasiam}.
A proof of the ``only-if'' part can be found, e.g., in \cite{Mmultexcstr17}.
\end{proof}

The exchange property for M-concave set functions
is in fact equivalent to a local exchange property
under some assumption on the effective domain.
We say that a family $\mathcal{B}$ of equi-cardinal subsets is 
{\em connected}
if, for any distinct $X, Y \in \mathcal{B}$,
there exist $i \in X \setminus Y$ and $j \in Y \setminus X$ such that
$Y + i - j \in \mathcal{B}$.
As is easily seen, $\mathcal{B}$ is connected
if and only if,
for any distinct $X, Y \in \mathcal{B}$ 
there exist distinct
$i_{1}, i_{2}, \ldots, i_{m} \in X \setminus Y$
and 
$j_{1}, j_{2}, \ldots, j_{m} \in Y \setminus X$,
where $m = |X \setminus Y| = |Y \setminus X|$,
such that 
$Y \cup \{ i_{1}, i_{2}, \ldots, i_{k} \} \setminus \{ j_{1}, j_{2}, \ldots, j_{k} \} \in \mathcal{B}$
for $k=1,2,\ldots,m$.

The following theorem is a strengthening
 by Shioura \cite[Theorem 2]{Shi00lev}
of the local exchange theorem of Dress--Wenzel \cite{DWperf92} and Murota \cite{Mmax97}
(see  also \cite[Theorem 5.2.25]{Mspr2000}, \cite[Theorem 6.4]{Mdcasiam})%
\footnote{
In \cite{DWperf92,Mmax97},
the effective domain is assumed to be a matroid basis family,
and the assumption is weakened to connectedness in \cite{Shi00lev}.
It is well known that a matroid basis family is connected.
}.  

\begin{theorem} \label{THmcavlocexc01}
A set function  $f: 2\sp{N} \to \RR \cup \{ -\infty \}$ 
is M-concave
if and only if 

{\rm (i)}
$\dom f$ is a connected nonempty family of equi-cardinal sets,
and 

{\rm (ii)}
for any $X, Y \in \dom f$ 
with $|X \setminus Y|=2$,
there exist some $i \in X \setminus Y$
and $j \in Y \setminus X$ for which {\rm (\ref{valmatexc1})} holds.
\end{theorem}
\begin{proof}
The ``only-if'' part is obvious.   For the ``if'' part,   
the proof of Theorem 5.2.25 in \cite[pp.295--297]{Mspr2000}
works with the only modification in the proof of Claim 2 there.
Since the proof is omitted in \cite{Shi00lev},
we include the proof in Section~\ref{SClocexcproof}.
\end{proof}

\section{A Proof of Theorem \ref{THmnatcavbyconjfn01}}
\label{SCproof}

We prove the characterization  of M$\sp{\natural}$-concavity
by submodularity of the conjugate function (Theorem \ref{THmnatcavbyconjfn01}).
Let 
$f: 2\sp{N} \to \RR \cup \{ -\infty \}$
be a set function with $\dom f \not= \emptyset$,
and $g: \RR\sp{N} \to \mathbb{R}$ be its conjugate function,
which is defined as
$g(p) = \max\{  f(X) - p(X)   \mid X \subseteq N \}$
in (\ref{conjcave2vex01}).

We first show that \Mnat-concavity of $f$ implies submodularity of $g$.

\begin{lemma} \label{LMmnatTOconjsubm01}
If $f$ is M$\sp{\natural}$-concave, then $g$ is submodular. 
\end{lemma}
\begin{proof}
As is well known, $g$ is submodular if and only if
\begin{equation} \label{mTOcjsbm01prf1}
g( p+ a \unitvec{i}) + g( p+ b \unitvec{j})  \geq 
g( p) + g( p + a \unitvec{i}+ b \unitvec{j}) 
\end{equation}
for any $p \in \RR\sp{N}$, distinct $i,j \in N$, and $a, b \geq 0$,
where $\unitvec{i}$ and $\unitvec{j}$ are the $i$th and $j$th unit vectors, respectively.
For simplicity of notation we assume $p=\veczero$,
and write
$p\sp{i}= a \unitvec{i}$,
$p\sp{j}= b \unitvec{j}$,
and
$p\sp{ij}= a \unitvec{i}+ b \unitvec{j}$.
Take $X, Y \subseteq N$ such that
\[
g(p) = f(X) - p(X) = f(X),
\quad
g( p\sp{ij}) = f(Y) - p\sp{ij}(Y) = f(Y) - a | Y \cap \{ i \} | - b | Y \cap \{ j \} |.
\]
Note also that
$g( p\sp{i}) = \max\ \{ f(Z)  - a | Z \cap \{ i \} |   \mid  Z \subseteq N \} $
and similarly for $g( p\sp{j})$.

\begin{itemize}
\item
If
$| Y \cap \{ i,j \} | =2$, then
$g(p)+g( p\sp{ij})  = ( f(X) - a ) + ( f(Y) - b )
\leq ( f(X) - a | X \cap \{ i \} | )
  + ( f(Y) - b | Y \cap \{ j \} | )
\leq g( p\sp{i}) + g( p\sp{j})$.

\item
If
$| Y \cap \{ i,j \} | =1$, we may assume $i \in Y$ and $j \not\in Y$. 
Then
$g(p)+g( p\sp{ij})  =   f(X)   +  ( f(Y) - a)
\leq ( f(X) - a| X \cap \{ i \} | )
  + ( f(Y) - b | Y \cap \{ j \} | )
\leq g( p\sp{i}) + g( p\sp{j})$.

\item
If
$| Y \cap \{ i,j \} | =0$, then
$g(p)+g( p\sp{ij})  =   f(X)  +  f(Y)$.
If $i \not\in X$, we have
$f(X)  +  f(Y)
= ( f(X) - a | X \cap \{ i \} | )
  + ( f(Y) - b | Y \cap \{ j \} | )
\leq g( p\sp{i}) + g( p\sp{j})$.
Similarly, if $j \not\in X$.
Suppose $\{ i,j \} \subseteq X$.
By the \Mnat-concave exchange property,
we have
$f(X)  +  f(Y) \leq f(X')  +  f(Y')$,
where 
$(X',Y')=(X-i, Y+i)$ or 
$(X',Y')=(X-i+k, Y+i-k)$ for some $k \in Y \setminus X$.
Since $i \not\in X'$ and $j \not\in Y'$,
we have
$f(X')  +  f(Y') 
=  ( f(X') - a| X' \cap \{ i \} | )  + ( f(Y') - b | Y' \cap \{ j \} | )
\leq g( p\sp{i}) + g( p\sp{j})$.
\vspace{-1.5\baselineskip}
\end{itemize}
\end{proof}
\vspace{\baselineskip}

Next, we show, in two steps, that submodularity of $g$ implies \Mnat-concavity of $f$.
We treat the M-concave case
in Lemmas \ref{LMmnatFRconjsubm01-0} to \ref{LMmnatFRconjsubm01-2},
and the \Mnat-concave case in Lemma~\ref{LMmnatFRconjsubm01-4}.
It is emphasized that the combinatorial essence is captured in Lemma~\ref{LMmnatFRconjsubm01-1}
for the M-concave case.

\medskip

\begin{lemma} \label{LMmnatFRconjsubm01-0}
If $\dom f$ is a family of equi-cardinal sets
and $g$ is submodular, 
then $\dom f$ is connected.
\end{lemma}
\begin{proof}
To prove this by contradiction,
suppose that $\dom f$ is not connected. 
Then there exist  $X, Y \in \dom f$ 
such that $|X \setminus Y|=|Y \setminus X| \geq 2$ and 
there exists no $Z \in \dom f \setminus \{ X, Y \}$ satisfying 
$X \cap Y \subseteq Z \subseteq X \cup Y$.
Let $i_{0}$ be any element of $X \setminus Y$
and $j_{0}$ be any element of $Y \setminus X$.
Let $M$ be a sufficiently large positive number
in the sense that $M \gg n$ and $M \gg F$ for 
$F = \max \{ |f(W)| \mid W \in \dom f \}$.
Define $p, q \in \RR^{N}$ by
\begin{align*} 
& p_{i}  =
   \left\{  \begin{array}{ll}
    -M           &   (i = i_{0}) ,     \\
    0           &   (i  \in (X \setminus Y) \setminus \{ i_{0} \}) ,     \\
                     \end{array}  \right.
\quad 
q_{i}  = 
   \left\{  \begin{array}{ll}
    0           &   (i = i_{0}) ,     \\
    -M           &   (i  \in (X \setminus Y) \setminus  \{ i_{0} \}) ;     \\
                     \end{array}  \right.
\end{align*}
\begin{align*} 
& p_{i} = q_{i}  =
   \left\{  \begin{array}{ll}
    -M           &   (i = j_{0}) ,     \\
    0           &   (i  \in (Y \setminus X) \setminus  \{ j_{0} \}) ,     \\
   - M\sp{2}     &   (i \in X \cap Y ),  \\
   + M\sp{2}     &   (i \in N \setminus (X \cup Y)  ) . \\
                     \end{array}  \right.
\end{align*}
Denote $m=|X \setminus Y|$ and $C = M\sp{2} |X \cap Y|$.
Since there is no $Z \in \dom f \setminus \{ X, Y \}$ satisfying 
$X \cap Y \subseteq Z \subseteq X \cup Y$,
we have
\begin{align*} 
g(p) &= 
\max\{ f(X) - p(X), f(Y)-p(Y) \} 
\notag \\ &= 
\max\{ f(X) + M, f(Y)+ M \} + C
\notag \\ &\leq
 F + M + C,
 \\
g(q) &= 
\max\{ f(X) - q(X), f(Y)-q(Y) \} 
\notag \\ &= 
\max\{ f(X) + (m-1)M, f(Y)+ M \} + C
\notag \\ &\leq
 F + (m-1)M + C,
\end{align*}
and therefore
\begin{equation} \label{domconnprfgpgq} 
g(p) + g(q ) \leq  2 F + mM + 2C.
\end{equation}
Similarly, we have
\begin{align*} 
g(p \vee q) &= 
\max\{ f(X) - (p \vee q) (X), f(Y)- (p \vee q)(Y) \} 
\notag \\ &= 
\max\{ f(X) , f(Y)+ M \} + C
\notag \\ & =
 f(Y)+ M + C,
 \\
g(p \wedge q ) &= 
\max\{ f(X) - (p \wedge q )(X), f(Y)-(p \wedge q )(Y) \} 
\notag \\ &= 
\max\{ f(X) + m M, f(Y)+ M \} + C
\notag \\ & =
 f(X) +  mM + C,
\end{align*}
and therefore
\begin{equation} \label{domconnprfgpVqgpAq} 
g(p \vee q) + g(p \wedge q ) =  f(X) + f(Y) + (m+1)M + 2C.
\end{equation}
Since $M \gg F$,  it follows from (\ref{domconnprfgpgq}) and (\ref{domconnprfgpVqgpAq}) that
$g(p) + g(q ) < g(p \vee q) + g(p \wedge q )$,
which contradicts the submodularity of $g$.
\end{proof}

\begin{lemma} \label{LMmnatFRconjsubm01-1}
If $\dom f$ is a family of equi-cardinal sets
and $g$ is submodular, 
then $f$ has the local exchange property
{\rm (ii)} in Theorem \ref{THmcavlocexc01}.
\end{lemma}

\begin{proof}
To prove by contradiction, suppose that
the local exchange property fails for 
\RED{%
some $X, Y \in \dom f$ 
}%
with $|X \setminus Y|=|Y \setminus X|=2$.
To simplify notations we assume
$X \setminus Y = \{ 1 , 2 \}$ and 
$Y \setminus X = \{ 3 , 4 \}$,
and write 
$\alpha_{ij} = f((X \cap Y) + i + j)$, etc.
Then we have
$\alpha_{12} + \alpha_{34} > 
       \max\{\alpha_{13} + \alpha_{24}, \alpha_{14} + \alpha_{23}\}$
by the failure of the local exchange property.
Consider an undirected graph $G =(V,E)$ on vertex set $V = \{ 1,2,3,4 \}$ 
and edge set $E = \{ (i,j) \mid \alpha_{ij} > -\infty \}$.
The graph $G$ has a unique maximum weight perfect matching $M = \{ (1,2), (3,4) \}$
with respect the edge weight $\alpha_{ij}$.
By duality (see Remark \ref{RMmnatconj01LPdual} below)
there exists 
$\hat p = (\hat p_{1}, \hat p_{2}, \hat p_{3}, \hat p_{4}) \in \RR\sp{4}$ such that
$\alpha_{12} = \hat p_{1} + \hat p_{2}$,
$\alpha_{34} = \hat p_{3} + \hat p_{4}$,
and
$\alpha_{ij} < \hat p_{i} + \hat p_{j}$
if $(i,j) \not= (1,2), (3,4)$.
Define 
$\beta_{ij} = \alpha_{ij} - \hat p_{i} - \hat p_{j}$,
to obtain
$\beta_{12} = \beta_{34} = 0$
and
$\beta_{ij} < 0$ if $(i,j) \not= (1,2), (3,4)$.
\RED{%
Let $B=\min \{ |\beta_{ij}| \mid (i,j) \ne (1,2), (3,4) \}$ $(>0)$.
}%

To focus on
\RED{%
 $I = \{ 1,2,3,4 \}$ 
}%
we partition $p$ into two parts as
$p = (p', p'')$ with
\RED{%
$p' \in \RR^{I}$ and $p'' \in \RR^{ N \setminus I }$.
}%
We express
$p' = \hat p + q$ 
with
\RED{%
$q = (q_{1}, q_{2}, q_{3}, q_{4}) \in \RR\sp{I}$,
}%
while fixing $p''$ to the vector $\bar p$ defined by
\begin{align*} 
\bar p_{i}  =
   \left\{  \begin{array}{ll}
   - M     &   (i \in X \cap Y ),  \\
   + M     &   (i \in N \setminus (X \cup Y) ) \\
                     \end{array}  \right.
\end{align*}
with a sufficiently large positive number $M$.
Let 
$h(q) = g(\hat p + q, \bar p) - M|X \cap Y|$.
By the choice of $\hat p$ and $\bar p$
as well as the assumed equi-cardinality of $\dom f$,
we have
\RED{%
\begin{align*} 
h(q) &= g(\hat p + q, \bar p) - M|X \cap Y|
\\
&=\max \{ f(Z) 
 - (\hat p + q)(Z \cap I) - \bar p(Z \setminus I) 
 \mid Z \subseteq  N  \} 
 - M|X \cap Y|
\\
&=\max \{ f(Z) 
 - (\hat p + q)(Z \cap I) 
 \mid (X \cap Y) \subseteq Z \subseteq (X \cap Y) \cup I \}
\\
&=\max \{ f((X \cap Y) \cup J) - (\hat p + q)(J)  \mid J \subseteq I, |J|=2  \}
\\
&= \max \{ \beta_{ij} - q_{i} - q_{j} \mid  i,j \in I, \ i \ne j \}
\end{align*}
}
if $\| q \|_{\infty}$ is small enough compared with $M$.
\RED{%
Furthermore, 
if 
$\| q \|_{\infty} \leq B/4$,
we have
\[
h(q)   = \max \{  - q_{1} - q_{2},  - q_{3} - q_{4} \}.
\] 
}%
Let $a$ 
\RED{%
be a (small) positive number satisfying $a \leq B/4$.
}%
Then 
$h(0,0,0,0) = h(a,-a,0,0) = h(a,0,0,0) = 0$ and
$h(0,-a,0,0) = a$.
This shows a violation of submodularity of $h$, and hence that of $g$.
\end{proof}

Lemmas \ref{LMmnatFRconjsubm01-0} and \ref{LMmnatFRconjsubm01-1}
with Theorem \ref{THmcavlocexc01}
show the following.

\medskip

\begin{lemma} \label{LMmnatFRconjsubm01-2}
If $\dom f$ is a nonempty family of equi-cardinal sets
and $g$ is submodular, 
then $f$ is an M-concave function.
\finbox
\end{lemma}

\medskip

\begin{remark} \rm  \label{RMmnatconj01LPdual}
In general, the perfect matching polytope of a graph $G=(V,E)$
is described by the following system of equalities for $x \in \RR\sp{E}$:
(i) $x_{e} \geq 0$ for each $e \in E$,
(ii) $x(\delta(v)) = 1$  for each $v \in V$,
(iii) $x(\delta(U)) \geq 1$  
for each $U \subseteq V$ with $|U|$ being odd $\geq 3$,
where $\delta(v)$ denotes the set of edges incident to a vertex $v$
and   $\delta(U)$ the set of edges between $U$ and $V \setminus U$;
see Schrijver \cite[Section 25.1]{Sch03}.
In the proof of Lemma~\ref{LMmnatFRconjsubm01-1} we have
$V = \{ 1,2,3,4 \}$, in which case the inequalities of type (iii) are not needed,
since $\delta(U)= \delta(v)$ for $U$ with $|U|=3$ and the vertex $v \in V \setminus U$.
Consider the maximum weight perfect matching problem on our $G=(V,E)$.
This problem can be formulated in a linear program
to maximize
$\sum_{(i,j) \in E} \alpha_{ij} x_{ij}$
subject to  $\sum_{j} x_{ij} = 1$
 for $i=1,2,3,4$
and
$x_{ij} \geq 0$ for $(i,j) \in E$.
Our assumption
$\alpha_{12} + \alpha_{34} > 
       \max\{\alpha_{13} + \alpha_{24}, \alpha_{14} + \alpha_{23}\}$
means that this problem has a unique optimal solution $x$ with
$x_{12}= x_{34}=1$ and $x_{ij}= 0$ for $(i,j) \not= (1,2), (3,4)$.
The dual problem is to minimize  
$p_{1} + p_{2} + p_{3} + p_{4}$
subject to
$p_{i} + p_{j} \geq \alpha_{ij}$ for $(i,j) \in E$.
The strict complementary slackness guarantees the
existence of a pair of optimal solutions $(x_{ij} \mid (i,j) \in E)$ and $(p_{i} \mid i=1,2,3,4)$
with the property that 
either $x_{ij} > 0$ or $p_{i} + p_{j} > \alpha_{ij}$ 
(exactly one of these) holds 
for each $(i,j) \in E$.
Therefore, there exists
$(\hat p_{1}, \hat p_{2}, \hat p_{3}, \hat p_{4})$ such that
$\alpha_{12} = \hat p_{1} + \hat p_{2}$,
$\alpha_{34} = \hat p_{3} + \hat p_{4}$,
and
$\alpha_{ij} < \hat p_{i} + \hat p_{j}$
for $(i,j) \not= (1,2), (3,4)$.
\finbox
\end{remark}

\medskip

Next we turn to the \Mnat-concave case.
Consider the function 
$\tilde{f}: 2^{\tilde N} \to \RR \cup \{ -\infty \}$
of (\ref{assocMdef})
associated with $f: 2^{N} \to \RR \cup \{ -\infty \}$,
where $\tilde{N} = N \cup S$
and $\dom \tilde{f} \subseteq \{ X \mid |X|=r \}$.
We take $S$ with $|S| \geq r-r'+2$.
Let $\tilde{g}(p,q)$ denote the conjugate of $\tilde{f}$,
where $p \in \RR\sp{N}$ and $q \in \RR\sp{S}$. 

\medskip

\begin{lemma} \label{LMmnatFRconjsubm01-4}
If $g$ is submodular, then $\tilde{g}$ is submodular.
\end{lemma}
\begin{proof}
By definition,
\begin{equation} \label{mFRconjsubm01prf3}
  \tilde{g}(p,q) = \max\{ f(X) - p(X) - q(U) \mid X \subseteq N,  \ U \subseteq S, \ |X| + |U| = r \} .
\end{equation}
It suffices to prove that
\begin{equation} \label{mTOcjsbm01prf6}
\tilde{g}( \tilde p + a \tilde{\unitvec{i}}) + \tilde{g}( \tilde p + b \tilde{\unitvec{j}}) 
\geq 
\tilde{g}(\tilde p)  + \tilde{g}( \tilde p + a \tilde{\unitvec{i}}+ b \tilde{\unitvec{j}}) 
\end{equation}
holds for any $\tilde p = (p,q) \in \RR\sp{N \cup S}$, distinct $i,j \in N  \cup S$, 
and $a, b \geq 0$,
where $\tilde{\unitvec{i}}$ and $\tilde{\unitvec{j}}$ are the $i$th and $j$th unit vectors
in $\RR\sp{N \cup S}$, respectively.
\RED{
Let
$h_{ij}(a,b) = \tilde{g}( \tilde p + a \tilde{\unitvec{i}} + b \tilde{\unitvec{j}})$.
Then \eqref{mTOcjsbm01prf6} can be rewritten as
\begin{equation} \label{mFRconjsubm01prf7}
h_{ij}(a,0) + h_{ij}(0,b)
\geq 
h_{ij}(0,0) + h_{ij}(a,b).
\end{equation}
In the following we assume
$\tilde p = (p,q) = \veczero$
for notational simplicity
(without essential loss of generality).
Then it follows from \eqref{mFRconjsubm01prf3} that
\begin{align} 
& h_{ij}(a,b) 
= \tilde{g}( \tilde p + a \tilde{\unitvec{i}} + b \tilde{\unitvec{j}})
\nonumber \\
&=  \max\{ f(X)  - a | (X \cup U) \cap \{ i \} | 
     - b | (X \cup U) \cap \{ j \} | \mid |X| + |U| = r \} .
\label{mFRconjsubm01prf8}
\end{align}
If $i, j \in S$, for example,
the function to be maximized reduces to
$f(X)  - a | U \cap \{ i \} | - b | U \cap \{ j \} |$
while, for each $X \in \dom f$,
there exists a subset $U$ of $S$ satisfying 
$|X| + |U| = r$ and $U \cap \{ i,j \} = \emptyset$ 
since $|S| \geq r-r'+2$ and $|X| \geq r'$.
Therefore, 
$h_{ij}(a,b) =   \max_{X} f(X) = g(\veczero) = g(p)$
for $i, j \in S$.
In a similar way 
we obtain
\[
h_{ij}(a,b) = 
\begin{cases}
g(p + a \unitvec{i}+ b \unitvec{j})
& (i,j \in N),
\\
g(p + a \unitvec{i})
& (i \in N, j \in S),
\\
g(p + b \unitvec{j})
& (i \in S, j \in N),
\\
g(p)
& (i,j \in S) .
\end{cases}
\]
Then \eqref{mFRconjsubm01prf7} follows from the assumed submodularity of $g$.
}
\OMIT{
For simplicity of notation we assume 
$\tilde p = \veczero$,
and write
$\tilde p\sp{i}= a \tilde{\unitvec{i}}$,
$\tilde p\sp{j}= b \tilde{\unitvec{j}}$,
and
$\tilde p\sp{ij}= a \tilde{\unitvec{i}}+ b \tilde{\unitvec{j}}$.
Take $X, Y \subseteq N$ 
and $U, V \subseteq S$ 
such that
$|X| + |U| = |Y| + |V| = r$, 
\begin{align*} 
&  \tilde{g}(\tilde p) = f(X) - \tilde p (X \cup U) = f(X),
\\
& \tilde{g}(\tilde p\sp{ij}) = f(Y) - \tilde p\sp{ij}(Y \cup V) 
= f(Y) - a | (Y \cup V) \cap \{ i \} | - b | (Y \cup V) \cap \{ j \} |.
\end{align*} 
Note also that
$\tilde{g}(\tilde p\sp{i}) = \max \{ f(Z)  - a | (Z \cup W) \cap \{ i \} | 
  \mid  Z \subseteq N, \ W \subseteq S \} $
and similarly for $\tilde g(\tilde p\sp{j})$.

If $\{ i,j \} \subseteq N$,
(\ref{mTOcjsbm01prf6}) reduces to 
$g( p+ a \unitvec{i}) + g( p+ b \unitvec{j})  \geq 
g( p) + g( p + a \unitvec{i}+ b \unitvec{j})$,
which holds since $g$ is assumed to be submodular.
The remaining cases are easier (not essential).

In case of $\{ i,j \} \subseteq S$,
we can assume, by $|S| \geq r-r'+2$, that $V \cap \{ i,j \} = \emptyset$,
which implies that
$\tilde{g}(\tilde p\sp{ij}) = g(p)$.
Similarly, we have
$\tilde{g}(\tilde p\sp{i}) = \tilde{g}(\tilde p\sp{j}) = g(p)$ as well as 
$\tilde{g}(\tilde p) = g(p)$.
Therefore, (\ref{mTOcjsbm01prf6}) holds.

In case of $| N \cap \{ i,j \} | = | S \cap \{ i,j \} | =1$, 
we may assume $i \in N$ and $j \in S$ by symmetry.
By $|S| \geq r-r'+2$, 
we have
$\tilde{g}(\tilde p\sp{ij}) = \tilde{g}(\tilde p\sp{i}) = g(p\sp{i})$
and
$\tilde{g}(\tilde p\sp{j}) = \tilde{g}(\tilde p) = g(p)$,
where
$p\sp{i}= a \unitvec{i} \in \RR\sp{N}$ and
$p\sp{j}= b \unitvec{j} \in \RR\sp{N}$.
Therefore, (\ref{mTOcjsbm01prf6}) holds.
}
\end{proof}

We are now in the position to complete the proof of Theorem \ref{THmnatcavbyconjfn01}.
If the conjugate function $g$ of $f$
is submodular, $\tilde{g}$ is also submodular by Lemma~\ref{LMmnatFRconjsubm01-4}.
Then $\tilde{f}$ is M-concave by Lemma~\ref{LMmnatFRconjsubm01-2},
and therefore $f$ is \Mnat-concave by Proposition~\ref{PRmnatequicardvalmat}.

\section{Appendix: Proof of Theorem \ref{THmcavlocexc01}}
\label{SClocexcproof}

A self-contained proof of Theorem \ref{THmcavlocexc01} is presented here.
This is basically the same as the proof of Theorem 5.2.25 in \cite[pp.295--297]{Mspr2000}
adapted to our present notation, with the difference only in the proof of Claim 2.

Let $\mathcal{B} = \dom f$.
For $p \in \RR\sp{N}$ we  define
\[
 f_{p}(X) = f(X)+p(X),
\quad
 f_{p}(X,i,j) = f_{p}(X-i+j)-f_{p}(X)
  \qquad (X \in \mathcal{B}),
\]
where $f_{p}(X,i,j)= - \infty$ if $X-i+j \not\in \mathcal{B}$. 
For $X, Y \in \mathcal{B}$, $i \in X\setminus Y$, and 
$j \in Y\setminus X$, we have
\begin{equation}  \label{wpeq}
 f(X,i,j) + f(Y,j,i)
  =  f_{p}(X,i,j) +  f_{p}(Y,j,i).
\end{equation}
If $X \in \mathcal{B}$, 
$X \setminus Y = \{ i_{0}, i_{1} \}$, 
$Y \setminus X = \{ j_{0}, j_{1} \}$
(with $i_{0} \not= i_{1}$, $j_{0} \not= j_{1}$), 
the local exchange property
(condition (ii) in Theorem \ref{THmcavlocexc01}) 
implies%
\footnote{
If $Y \not\in \mathcal{B}$,
the inequality (\ref{vmprupper2}) is trivially true with $ f_{p}(Y)  = -\infty$.
} 
\begin{equation} 
 f_{p}(Y) - f_{p}(X)  \leq 
 \max \{ f_{p}(X,i_{0},j_{0}) + f_{p}(X,i_{1},j_{1}),  \   
        f_{p}(X,i_{0},j_{1}) + f_{p}(X,i_{1},j_{0}) \} .
\label{vmprupper2}
\end{equation}

\medskip

Define
\begin{eqnarray*}
 \mathcal{D} &=&
   \{ (X,Y) \mid X, Y \in \mathcal{B}, \ 
    \exists \,  i_{*} \in X \setminus Y, \  \forall \,  j \in Y\setminus X:
\nonumber \\
 & & \qquad\qquad \ 
    f(X)+f(Y) > f(X-i_{*}+j) + f(Y+i_{*}-j)
   \},
\end{eqnarray*}
which denotes the set of pairs $(X,Y)$ for which
the exchange property (\ref{valmatexc1}) fails.
We want to show $\mathcal{D} = \emptyset$.

Suppose, to the contrary, that 
$\mathcal{D} \not= \emptyset$, and take
$(X,Y) \in \mathcal{D}$ such that $|Y\setminus X|$ is minimum and let
$i_{*} \in X \setminus Y$ be the element in the definition of $\mathcal{D}$.
We have $|Y\setminus X| > 2$.
Define $p \in \RR\sp{N}$ by
\[
 p_{j} = 
   \left\{  \begin{array}{ll}
     - f(X,i_{*},j)  &   (j \in Y\setminus X, \  X-i_{*}+j \in \mathcal{B}) ,      \\
     f(Y,j,i_{*}) + \varepsilon
     &   (j \in Y\setminus X, \  X-i_{*}+j \not\in \mathcal{B}, \  Y+i_{*}-j \in \mathcal{B}) ,  \\
     0  & (\mbox{otherwise})
                     \end{array}  \right.
\]
with some $\varepsilon > 0$.
\begin{description}
\item[Claim 1:] 
\quad     \vspace*{-\baselineskip}
\begin{eqnarray}
 f_{p}(X,i_{*},j) &=&  0
  \qquad \mbox{if $j \in Y\setminus X, \  X-i_{*}+j \in \mathcal{B}$},
 \label{wpB1}
 \\
 f_{p}(Y,j,i_{*}) & < & 0
  \qquad \mbox{for  $j \in Y\setminus X$}.
 \label{wpB2}
\end{eqnarray}
\end{description}
The  inequality (\ref{wpB2})  can be shown as follows.
If  $X-i_{*}+j \in \mathcal{B}$, we have
$f_{p}(X,i_{*},j)=0$ by (\ref{wpB1}) and 
\[
 f_{p}(X,i_{*},j) + f_{p}(Y,j,i_{*}) 
= f(X,i_{*},j) + f(Y,j,i_{*}) < 0
\]
by (\ref{wpeq}) and the definition of $i_{*}$.
Otherwise we have
$f_{p}(Y,j,i_{*})= - \varepsilon$ or $-\infty$
according to whether $Y+i_{*}-j \in \mathcal{B}$ or not.
\begin{description}
\item[Claim 2:] 
There exist $i_{0} \in X \setminus Y$ and  $j_{0} \in Y\setminus X$ 
such that
$i_{0} \not= i_{*}$, $Y+i_{0}-j_{0} \in \mathcal{B}$, and
\begin{equation} \label{wpmax}
 f_{p}(Y,j_{0},i_{0}) \geq f_{p}(Y,j,i_{0})
  \qquad (j \in Y\setminus X) .
\end{equation}
\end{description}

First, we show the existence of
$i_{0} \in X \setminus Y$ and $j \in Y\setminus X$
such that $Y + i_{0} - j \in \mathcal{B}$ and $i_{0} \not= i_{*}$.
By connectedness of $\mathcal{B}$
and $|X \setminus Y| > 2$, 
there exist $i_{1} \in X \setminus Y$ and $j_{1} \in Y \setminus X$ such that
$Z=Y + i_{1} - j_{1} \in \mathcal{B}$.
If $i_{1} \not= i_{*}$,  
we are done with $(i_{0},j) = (i_{1},j_{1})$. 
Otherwise, 
again by connectedness, 
there exist $i_{2} \in X \setminus Z$ and $j_{2} \in Z \setminus X$ such that
$W = Z + i_{2} - j_{2} \in \mathcal{B}$.
Since $|W \setminus Y| = 2$
with
$W = Y + \{ i_{1}, i_{2} \} - \{ j_{1}, j_{2} \}$,
we obtain
$ Y + i_{2} - j_{1} \in \mathcal{B}$
or
$ Y + i_{2} - j_{2} \in \mathcal{B}$
from (\ref{valmatexc1}).
Hence we can take $(i_{0},j) =(i_{2},j_{1})$ or $(i_{0},j) =(i_{2},j_{2})$;
note that $i_{2}$ is distinct from $i_{*}$.
Next we choose the element $j_{0}$.
By the choice of $i_{0}$, we have $f_{p}(Y,j,i_{0}) > -\infty$ 
for some $j \in Y\setminus X$.
By letting $j_{0}$ to be an element
$j \in Y\setminus X$ that maximizes 
$f_{p}(Y,j,i_{0})$, we obtain (\ref{wpmax}). 
Thus Claim 2 is established under the connectedness assumption.

\begin{description}
\item[Claim 3:]
$(X,Z) \in \mathcal{D}$ with $Z=Y+i_{0}-j_{0}$.
\end{description}
To prove this  it suffices to  show
\[
 f_{p}(X,i_{*},j) + f_{p}(Z,j,i_{*}) < 0
  \qquad (j \in Z\setminus X).
\]
We may restrict ourselves to $j$ with  $X-i_{*}+j \in \mathcal{B}$, 
since otherwise
the first term $f_{p}(X,i_{*},j)$ is equal to $-\infty$.
For such $j$  the first term is equal to zero by (\ref{wpB1}).
For the second term it follows from
(\ref{vmprupper2}), (\ref{wpB2}), and  (\ref{wpmax}) that
\begin{eqnarray*}
  f_{p}(Z,j,i_{*}) 
 & = &
 f_{p}(Y+\{ i_{0},i_{*} \} - \{ j_{0},j \} ) 
  -  f_{p}(Y+ i_{0} - j_{0})
 \nonumber \\
 & \leq &  \max \left[
   f_{p}(Y,j_{0},i_{0})+f_{p}(Y,j,i_{*}),
   f_{p}(Y,j,i_{0})+f_{p}(Y,j_{0},i_{*})
               \right] 
  -  f_{p}(Y,j_{0},i_{0})
 \nonumber \\
 &  < &  \max \left[
   f_{p}(Y,j_{0},i_{0}),   f_{p}(Y,j,i_{0})
               \right] 
  - f_{p}(Y,j_{0},i_{0})
 \  = 0.
\end{eqnarray*}

Since $|Z\setminus X| = |Y\setminus X|-1$, Claim 3 contradicts our choice
of $(X,Y) \in \mathcal{D}$.
Therefore we conclude  $\mathcal{D} = \emptyset$.
This completes the proof of Theorem \ref{THmcavlocexc01}.

\medskip

\begin{remark} \rm \label{RMproofmodif}
For the ease of reference,
we describe here the necessary change 
in the proof of \cite[Theorem 5.2.25]{Mspr2000} in the notation there. 
The necessary change is localized to the proof of 
\begin{quote}
Claim 2: 
There exist $u_{0} \in B\setminus B'$ and  $v_{0} \in B'\setminus B$ 
such that
$u_{0} \not= u_{*}$, $B'+u_{0}-v_{0} \in \mathcal{B}$, 
\begin{equation} \label{wpmax-1}
 \omega_{p}(B',v_{0},u_{0}) \geq \omega_{p}(B',v,u_{0})
  \qquad (v \in B'\setminus B) .
\end{equation}
\end{quote}
We now assume connectedness of $\mathcal{B}$, instead of its exchange property.
First, we show the existence of
$u_{0} \in B\setminus B'$ and $v \in B'\setminus B$
such that $B' + u_{0} - v \in \mathcal{B}$ and $u_{0} \not= u_{*}$.
By connectedness of $\mathcal{B}$
and $|B\setminus B'| > 2$, 
there exist $u_{1} \in B \setminus B'$ and $v_{1} \in B' \setminus B$ such that
$B''=B' + u_{1} - v_{1} \in \mathcal{B}$.
If $u_{1} \not= u_{*}$,  
we are done with $(u_{0},v) = (u_{1},v_{1})$. 
Otherwise, 
again by connectedness, 
there exist $u_{2} \in B \setminus B''$ and $v_{2} \in B'' \setminus B$ such that
$B''' = B'' + u_{2} - v_{2} \in \mathcal{B}$.
Since $|B''' \setminus B'| = 2$
with
$B''' =  B' + \{ u_{1}, u_{2} \} - \{ v_{1}, v_{2} \}$,
we obtain
$ B' + u_{2} - v_{1} \in \mathcal{B}$
or
$ B' + u_{2} - v_{2} \in \mathcal{B}$
from (\ref{valmatexc1}).
Hence we can take $(u_{0},v) =(u_{2},v_{1})$ or $(u_{0},v) =(u_{2},v_{2})$;
note that $u_{2}$ is distinct from $u_{*}$.
Next we choose the element $v_{0}$.
By the choice of $u_{0}$, we have $\omega_{p}(B',v,u_{0}) > -\infty$ 
for some $v \in B'\setminus B$.
By letting $v_{0}$ to be an element
$v \in B'\setminus B$ that maximizes 
$\omega_{p}(B',v,u_{0})$, we obtain (\ref{wpmax-1}). 
Thus Claim 2 is established under the connectedness assumption.
\finbox
\end{remark}

\section*{Acknowledgement}
The authors are grateful to Satoru Fujishige for kindly explaining the proof of \cite{DL01gs}
written in Russian.
This work is supported by The Mitsubishi Foundation, CREST, JST, 
Grant Number JPMJCR14D2, Japan, and
JSPS KAKENHI Grant Numbers  26280004, 15K00030.

\end{document}